\documentclass{amsart}
\usepackage{amssymb}

\usepackage{color}
\usepackage{graphics}
\usepackage{amsmath,amscd}
\usepackage{amsthm}
\usepackage[all,cmtip]{xy}
\usepackage{tikz}
\usepackage{graphicx}
\usepackage{epstopdf}
\usepackage{mathtools}
\theoremstyle{definition}
\newtheorem{exmp}{Example}[section]
\usepackage[colorlinks=green,linkcolor=blue,citecolor=red,urlcolor=red]{hyperref}

\newtheorem{proposition}{Proposition}[section]
\newtheorem{theorem}[proposition]{Theorem}
\newtheorem{definition}[proposition]{Definition}
\newtheorem{lemma}[proposition]{Lemma}

\newtheorem{corollary}[proposition]{Corollary}

\begin{document}
	
	\title{On open book embedding of contact manifolds in the standard contact sphere}
	
	

	\author{Kuldeep Saha}
	\address{Chennai Mathematical Institute, Chennai, India}
	\email{kuldeep@cmi.ac.in}
	
	\begin{abstract}
		We prove some open book embedding results in the contact category with a constructive ap-
		proach. As a consequence, we give an alternative proof of a Theorem of Etnyre and Lekili that produces a large
		class of contact $3$-manifolds admitting contact open book embeddings in the standard contact $5$-sphere.
		We also show that all the Ustilovsky $(4m + 1)$-spheres contact open book embed in the standard contact $(4m + 3)$-sphere.
	\end{abstract}
	
	\maketitle

\section{Introduction}

An open book decomposition of a manifold $M^m$ is a pair $(V^{m-1},\phi)$, such that $M^m$ is diffeomorphic to $\mathcal{MT}(V^{m-1}, \phi) \cup_{id} \partial V^{m-1} \times D^2 $. Here, $V^{m-1}$, the \emph{page}, is a manifold with boundary, and $\phi$, the \emph{monodromy}, is a diffeomorphism of $V^{m-1}$ that restricts to identity in a neighborhood of the boundary $\partial V$. $\mathcal{MT}(V^{m-1}, \phi)$ denotes the mapping torus of $\phi$. We denote an \emph{open book}, with page $V^{m-1}$ and monodromy $\phi$, by $\mathcal{A}ob(V,\phi)$. The existence of open book decompositions, for a fairly large class of manifolds, is now known due the works of Alexander \cite{Al}, Winkelnkemper \cite{Wi}, Lawson \cite{La}, Quinn \cite{Qu} and Tamura \cite{Ta}. In particular, every closed, orientable, odd dimensional manifold admits an open book decomposition. 

Thurston and Winkelnkemper \cite{TW} have shown that starting from an exact symplectic manifold $(\Sigma^{2m},\omega)$ as page and a boundary preserving symplectomorphism $\phi_s$ of $(\Sigma^{2m},\omega)$ as monodromy, one can produce a contact $1$-form $\alpha$ on $N^{2m+1} = \mathcal{A}ob(\Sigma^{2m},\phi_s)$. A remarkable result of Giroux \cite{Gi} says that the converse is also true. Roughly speaking, this means that every contact manifold $(M^{2m+1},\xi)$, can be obtained by the construction of Thurston and Winkelnkemper, starting with some exact symplectic manifold $(V^{2m},\omega_V)$ as page and some boundary preserving symplectomorphism $\phi_V$ of $(V^{2m},\omega_V)$ as monodromy. We say, the contact manifold $(M^{2m+1},\xi)$ is \emph{supported} by the open book $\mathcal{A}ob(V^{2m},\omega_V,\phi_V)$. This gives a version of existence of open book in the setting of contact manifolds, and allows us to translate questions about contact manifolds into questions about \emph{contact open book decompositions} of that manifold. 

A closed manifold $M^{2n+1}$ \emph{open book embeds} into another closed manifold $V^{2N+1}$, if there is an open book decomposition of $V^{2N+1}$ and an embedding $\iota : M \rightarrow V$, such that the open book on $V^{2N+1}$ induces an open book decomposition on $\iota(M^{2n+1})$. One can then ask: When does a contact manifold $(M^{2n+1},\xi)$ embeds into another contact manifold $(V^{2N+1},\eta)$, such that there is a supporting open book decomposition of $(V^{2N+1},\eta)$, that induces a supporting open book on the embedded manifold? For a precise definition of contact open book embedding, see definition \ref{contact ob embed def}. 

The first results in contact open book embedding were obtained by Mori \cite{Mo} and Torres \cite{To}, using
techniques from approximate holomorphic geometry. In recent times much progress has been made on the
question of co-dimension $2$ contact embedding due to the works of Kasuya \cite{Ka2}, Etnyre and Furukawa \cite{EF},
Etnyre and Lekili \cite{EL} and Pancholi and Pandit \cite{PP}. Recently, the existence and uniqueness questions for co-dimension $2$ iso-contact embedding has been completely answered by the works of Casals, Pancholi and Presas \cite{CPP}, Casals and Etnyre \cite{CE} and Honda and Huang \cite{HH}. On the other hand, explicit constructions of smooth open book embeddings of closed, oriented $3$-manifolds into $S^2 \times S^3$ and $S^2 \tilde{\times} S^3$ has been found in \cite{PPS}. In this article, we prove some results regarding contact open book embedding in the standard contact sphere, with similar constructive approach. As an application, we give an alternate proof of a result due to Etnyre and Lekili \cite{EL} that provides contact open book embedding of a large class of contact $3$-manifolds in $(S^5,\xi_{std})$. Another application shows that all Ustilovsky spheres have co-dimension $2$ contact open book embeddings in the standard contact sphere (see section \ref{exotic}). Explicit co-dimension $2$ contact open book embeddings were previously constructed by Casals and Murphy \cite{CM} in order to construct embedding of an overtwisted contact sphere $(S^{2n-1}, \xi_{ot})$ into any contact manifold of dimension $2n + 1$.

 Throughout the article we assume all contact manifolds to be co-oriented and closed. Unless stated otherwise, we will denote a contact manifold $M$ with a contact hyperplane distribution $\xi$ on it by $(M,\xi)$. The standard contact structure on
 the unit sphere $S^{2m+1} \subset \mathbb{R}^{2m+2}$ will be denoted by $\xi_{std}$. For related definitions, we refer to section \ref{prelim}. For related notions in basic contact topology, we refer to \cite{Et} and \cite{Ge}.

\

We start with the following theorem.

\begin{theorem} \label{1st open book theorem}
	For $n\geq1$ and $k,l \in \mathbb{Z}$, $\mathcal{A}ob(DT^*S^n,d\lambda^n_{can},\tau_k)$ contact open book embeds in $\mathcal{A}ob(DT^*S^{n+1},d\lambda^{n+1}_{can},\tau_l)$.

Here, $d\lambda^n_{can}$ denotes the canonical symplectic form on the unit disk cotangent bundle $DT^*S^n$ of $S^n$, and $\tau_m$ denotes the $m$-fold Dehn-Seidel twist (section \ref{dehn twist}) for $m \in \mathbb{Z}$.

\end{theorem}

Theorem \ref{1st open book theorem} can also be obtained from the construction of contact open book embedding
of $\mathcal{A}ob(DT^*S^n, d\lambda_{can}^n, \tau_k)$ in $\mathcal{A}ob(D^{2n+2},d\lambda_0, id)$, due to Casals and Murphy \cite{CM}. For that we have to
stabilize the target page $(D^{2n+2},d\lambda_0)$, away from the image of $DT^*S^n$ under the embedding. The proof given in this article has a different approach, and will come in handy while proving Theorem \ref{2nd open book theorem} below.

In \cite{CMP}, Casals, Murphy and Presas gave a characterization of overtwisted contact structures in terms of open books. They showed that every overtwisted contact structure is a negative stabilization of some open book decomposition. In particular, $\mathcal{A}ob(DT^*S^n,d\lambda_{can},\tau_{-1})$ gives an overtwisted contact structure on $S^{2n+1}$. Thus, an immediate corollary of Theorem \ref{1st open book theorem} is the following fact, which was first shown by
Casals and Murphy \cite{CM}.

\begin{corollary}\label{1st open book corollary}
	For all $n\geq1$, there exists an overtwisted contact structure on $S^{2n+1}$ that contact open book embeds in $(S^{2n+3},\xi_{std})$.
\end{corollary}

Using Theorem \ref{1st open book theorem}, we can find a large class of contact manifolds that admit co-dimension $2$ contact open book embeddings in the standard contact sphere. For definition of plumbing and boundary connected sum, see section \ref{stabilization} and section \ref{bandsum}.  The boundary connected sum and plumbing will be denoted by $\#_b$ and $\mathsection$ respectively. 

\begin{definition}
	Consider the canonical symplectic structure $d\lambda_M$ on the cotangent bundle of a manifold $M$. We call a contact open book $\mathcal{A}ob(V^{2n},\omega,\phi)$ of \textit{type-$1$}, if it satisfies the following properties.
	
	\begin{enumerate}
		\item $(V^{2n},\omega)$ is symplectomorphic to $$(DT^*M_1 \mathsection DT^*M_2 \mathsection...\mathsection DT^*M_p\#_b DT^*N_1\#_b DT^*N_2....\#_b DT^*N_q, d\lambda_{M_1} \mathsection d\lambda_{M_2} \mathsection...\mathsection d\lambda_{M_p} \#_b d\lambda_{N_1} \#_b d\lambda_{N_2}\#_b...\#_b d\lambda_{N_q})$$.
		Here, $M_i$s and $N_j$s will always be either the equator sphere $S^n \subset S^{n+1}$, or a closed $n$-dimensional submanifold of $S^{n+1}$.
		
		\item The monodromy $\phi$ is generated by Dehn-Seidel twists along the $S^n$s among $M_i$s and $N_j$s . 
	\end{enumerate} 
	
\end{definition}

\begin{theorem} \label{2nd open book theorem}
	If $(M^{2n+1},\xi)$ is a contact manifold supported by an open book of type-$1$, then $(M^{2n+1},\xi)$ has a contact open book embedding in $(S^{2n+3},\xi_{std})$.
\end{theorem}

The next corollary gives an application of Theorem \ref{2nd open book theorem} to contact open book embedding of $3$-manifolds in $(S^5,\xi_{std})$. It was first proved by Etnyre and Lekili (Theorem $4.3$, \cite{EL}).

\begin{figure} \label{humphreygen}
	\centering
	\includegraphics{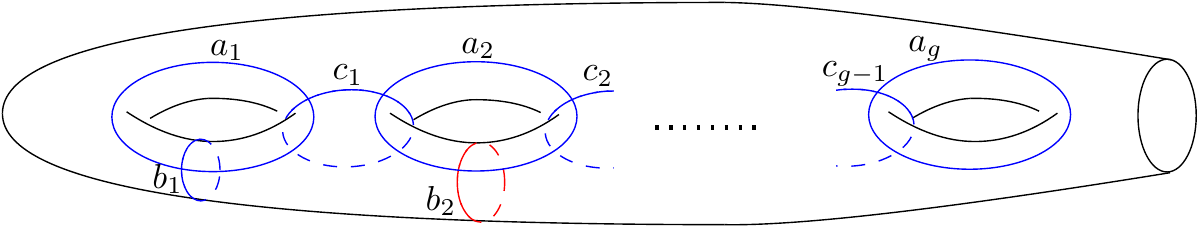}
	\caption{The Humphreys generators of mapping class groups of $\Sigma_g$ consists of the $(2g+1)$ curves given by $a_1,c_1,a_2,c_2,...a_{g-1},c_{g-1},a_g,b_1,b_2$.}

\end{figure}

\begin{corollary}\label{2nd open book corollary}
	Let $(M^3,\xi)$ be a contact $3$-manifold supported by an open book with page $\Sigma_g$ as in Figure \ref{humphreygen}. Say, the monodromy of the supporting open book is generated by Dehn twists along the blue curves: $b_1,a_1,c_1,a_2,c_2,...,a_{g-1},c_{g-1},a_g$. Then $(M^3,\xi)$ contact open book embeds in $(S^5,\xi_{std})$. 
\end{corollary}

	By Giroux \cite{Gi}, if two contact structures on $M^3$ are supported by the same open book, then the contact structures are contactomorphic. Moreover, any two open books supporting a contact structure are related by stabilizations \cite{Gi}. Given any supporting open book of a contact $3$-manifold, we can stabilize it finitely many times until the resulting open book has a page homeomorphic to $\Sigma_g$. Stabilization does not change the contactomorphism type of the manifold.
     
    Note that the relative mapping class group $\mathcal{MCG}(\Sigma_g, \partial \Sigma_g)$ is generated by Dehn twists along the curves $a_1, c_1,..., a_{g-1}, c_{g-1}, a_g, b_1, b_2$ and a curve $d$,
	parallel to the boundary $\partial \Sigma_g$. But, $\tau_d = (\tau_{b_1} \circ \tau_{a_1} \circ \tau_{c_1} \circ...\tau_{a_{g-1}} \circ \tau_{c_{g-1}} \circ \tau_{a_g})^{4g+2}$ (see Proposition $4.12$ in
	\cite{FM}). Thus, $\mathcal{MCG}(\Sigma_g, \partial \Sigma_g)$ can be generated by Dehn twists along the curves $a_1, c_1,..., a_{g-1}, c_{g-1}, a_g, b_1, b_2$. Every contact open book embedding also gives an iso-contact embedding. This
	shows that the only class of $3$-dimensional contact open books with page $\Sigma_g$, which may not admit iso-contact embedding in $(S^5, \xi_{std})$, are the ones with monodromy involving Dehn twists along the red curve $b_2$ in Figure \ref{humphreygen}. 
	
	We also note that a contact manifold $(M^3,\xi)$ may have a supporting open book with page $\Sigma_g$ and monodromy involving Dehn twists along $b_2$, and still can iso-contact embed in $(S^5,\xi_{std})$. Examples of such open books are discussed in section \ref{nonexample}.

\subsection{Acknowledgement} The author is grateful to Dishant M. Pancholi for his help and support during this work. He thanks John B. Etnyre for helpful comments. He also thanks the referee for valuable comments and suggestions, which helped improve this article. During this work, the author was supported by the National Board of Higher Mathematics, DAE, Govt. of India (grant no. 2/39(18)/2013/R and D/15916).

\section{preliminaries}\label{prelim}

\subsection{Contact open book and contact open book embedding} \label{contact open book}

We start with a contact analogue of the abstract open book decomposition. This construction of an \emph{abstract contact open book} is due to Thurston and Winkelnkemper \cite{TW}. The discussion here is based on the lecture notes by Otto Van Koert \cite{Ko}.

\

Let $(V,\partial V,d\alpha)$ be an exact symplectic manifold, with a collar neighborhood $N(\partial V)$ symplectomorphic to $((-1,0] \times \partial V,d(e^t \cdot \alpha))$, for $t \in (-1,0]$. The Liouville vector field $Y$ for $d \alpha$ is defined by $i_Yd\alpha = \alpha$. So, near boundary it looks like $\frac{\partial}{\partial t}$ and is transverse to $\partial V$, pointing outwards. The $1$-form $e^t \cdot \alpha$ induces a contact structure on $\partial V$. Let $\phi$ be a symplectomorphism of $(V,d\alpha)$ such that $\phi$ is identity in a neighborhood of the boundary. The following lemma, due to Giroux, shows that we can assume $\phi^*\alpha - \alpha$ to be exact.

\begin{lemma}[Giroux] \label{giroux lemma}
	The symplectomorphism $\phi$ of $(V,d\alpha)$ is isotopic, via symplectomorphisms which are identity near $\partial V$, to a symplectomorphism $\phi_1$ such that $\phi_1^*\alpha - \alpha$ is exact.
\end{lemma}

For a proof of the above lemma see \cite{Ko}.

Let $\phi^*\alpha - \alpha = dh$. Here, $h: V \rightarrow \mathbb{R}$ is a function well defined up to addition by constants. Note that $dt + \alpha$ is a contact form on $\mathbb{R} \times V$, where the $t$ coordinate is along $\mathbb{R}$. Consider the mapping torus $\mathcal{MT}(V,\phi)$ defined by the following map.

\begin{alignat*}{2}
\Delta: (\mathbb{R} \times V, dt + \alpha) &\longrightarrow& (\mathbb{R} \times V, dt + \alpha) \\
(t,x) &\longmapsto& (t-h,\phi(x)) 
\end{alignat*}

The contact form $dt + \alpha$ then descends to a contact form $\lambda$ on $\mathcal{MT}(V,\phi)$. Since $\phi$ is identity near $\partial V$, a contact neighborhood of the boundary of $\mathcal{MT}(V,\phi)$ looks like $((-\frac{1}{2},0) \times \partial V \times S^1, e^r\cdot\alpha|_{\partial V} + dt)$. Let $A(r, R) = \{z \in \mathbb{C} \ | \ r < |z| < R\}$. Define $\Phi$ as follows.

\begin{alignat*}{2}
\Phi: \partial V \times A(\frac{1}{2},1) &\longrightarrow& (-\frac{1}{2},0) \times \partial V \times S^1 \\
(v, r e^{it}) &\longmapsto& (\frac{1}{2} - r, v, t) 
\end{alignat*}

Using $\Phi$, we can glue $\mathcal{MT}(V,\phi)$ and $\partial V \times D^2$ along a neighborhood of their boundary, such that under $\Phi$, the $1$-form $\lambda$ pulls back to $(e^{\frac{1}{2}- r} \cdot \alpha|_{\partial V} + dt)$ on $\partial V \times A(\frac{1}{2}, 1)$. We want to extend this $1$-form to a form $\beta = h_1(r)\cdot \alpha|_{\partial V} + h_2(r)\cdot dt$, such that $\beta$ is contact in the interior of $\partial V \times D^2$. We can choose the functions $h_1$ and $h_2$ (see Figure $2$) so that $\beta$ becomes a globally defined contact form on $W^{2n+1} = \mathcal{MT}(V,\phi) \cup_{id} \partial V \times D^2$, and it coincides with $dt + \alpha$ on $\mathcal{MT}(V,\phi)$ and with $\alpha + r^2 dt$ on $\partial V \times D^2$. We will denote the resulting contact manifold $(W^{2n+1},\beta)$ as $\mathcal{A}ob(V,d\alpha;\phi)$.

\begin{figure}\label{h1 n h2}
	\begin{tikzpicture}
	\draw[thick,->] (0,0) -- (5,0) node[anchor=north east] {$r$};
	\draw[thick,->] (0,0) -- (0,4) node[anchor=north east] {$h_1(r)$}; 
	\draw (3.5 cm,1pt) -- (3.5 cm,-1pt) node[anchor=north] {$\frac{1}{2}$};
	\draw[very thick] (0,3) -- (0.5,3);
	\draw[very thick] (0.5,3) parabola (2,2.5);
	\draw[very thick] (4.2,1.5) parabola (2,2.5);
	\draw[black,thick,dashed] (3.5,0) -- (3.5,4);
	\draw (0,3) node[anchor=east]{$h_1(0)$};
	\draw[black,thick,dashed] (0,3) -- (4,3);

	\draw[thick,->] (8,0) -- (13,0) node[anchor=north east] {$r$};
	\draw[thick,->] (8,0) -- (8,4) node[anchor=north east] {$h_2(r)$}; 
	\draw (11.5 cm,1pt) -- (11.5 cm,-1pt) node[anchor=north] {$\frac{1}{2}$};
	\draw[very thick] (8,0) -- (8.5,0);
	\draw[very thick] (8.5,0) parabola (9.5,2);
	\draw[very thick] (10.5,3) parabola (9.5,2);
	\draw[very thick] (10.5,3) -- (12,3);
	\draw[black,thick,dashed] (11.5,0) -- (11.5,3.5);
	\draw[black,thick,dashed] (8,3) -- (12,3);
	
	\end{tikzpicture}	
	\caption[]{Functions for the contact form near binding}
\end{figure}
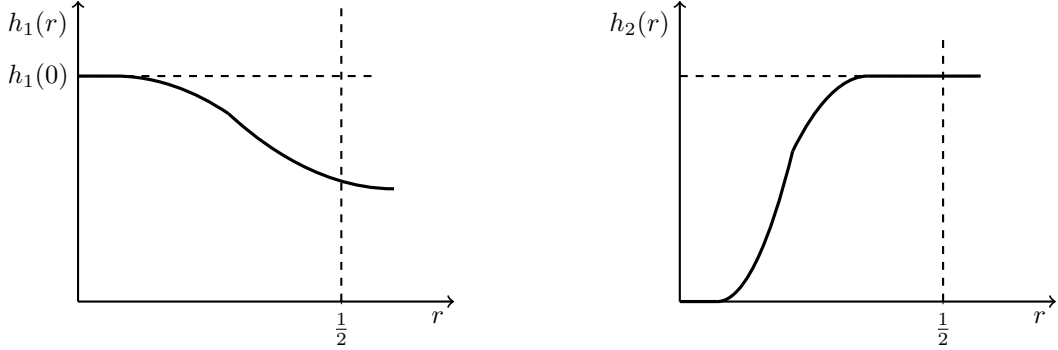

The contact manifold $\mathcal{A}ob(V,d\alpha,\phi)$ depends on the symplectic isotopy class of the monodromy $\phi$. There is an important symmetry property for open books. Let $(V,d\alpha)$ be as above. If $\phi_1$ and $\phi_2$ are symplectomorphisms of $(V,d\alpha)$, then $\mathcal{A}ob(V,d\alpha,\phi_1 \circ \phi_2)$ is contactomorphic to $\mathcal{A}ob(V,d\alpha,\phi_2 \circ \phi_1)$.

\begin{definition}[Contact open book]
	$\mathcal{A}ob(V, d\alpha,\phi)$ is called a \emph{contact open book} with \emph{page} $(V,d\alpha)$ and \emph{monodromy} $\phi$. $(\partial V,\alpha)$ is called the \emph{binding}. 
\end{definition}

Given a contact manifold $(M,\alpha)$ with a contact $1$-form $\alpha$, if one can find an open book $\mathcal{A}ob(V_M,\phi_M)$ of $M$, such that $d\alpha$ restricts to a symplectic form on $V_M$ and $\alpha$ induces, positive orientation on $M$ and positive contact orientation on $\partial V_M$, then one says that $\mathcal{A}ob(V_M,\phi_M)$ is an open book decomposition of $M$ \emph{supporting} the contact form $\alpha$.

If a contact manifold $(M,\xi)$ has a contact form $\alpha$ representing $\xi$, such that $\alpha$ has a supporting open book, then we say that $(M,\xi)$ has a \emph{supporting open book}. Sometimes, we may write $(M,\xi) = \mathcal{A}ob(V_M,d\alpha_M;\phi_M)$ to say that $(M,\xi)$ is $supported$ by the open book with page $(V_M,d\alpha_M)$ and monodromy $\phi_M$. Giroux \cite{Gi} then shows that any contact manifold $(M,\xi)$ has a supporting open book.

\begin{exmp}
	$(S^{2n+1},\xi_{std})$ has a contact open book decomposition with page $(D^{2n},\Sigma_{i=1}^n r_idr_id\theta_i)$ and monodromy identity.  
\end{exmp}

\begin{definition}[contact open book embedding]\label{contact ob embed def}
	$(M_1,\xi_1)$ \emph{contact open book embeds} into $(M_2,\xi_2)$, if there exist supporting contact open books $\mathcal{A}ob(\Sigma_1,d\alpha_1,\phi_1)$ and $\mathcal{A}ob(\Sigma_2,d\alpha_2,\phi_2)$, for $(M_1^{2n+1},\xi_1)$ and $(M_2^{2n+1},\xi_2)$ respectively, such that the following conditions hold.
	
	\begin{enumerate}
		\item There exists a proper symplectic embedding $g : (\Sigma_1,d\alpha_1) \rightarrow (\Sigma_2,d\alpha_2)$, i.e., $g^*d\alpha_2 = d\alpha_1$ ,
		\item $g \circ \phi_1  = \phi_2 \circ g$.
	\end{enumerate}  
\end{definition}

An embedding $\iota$, of a contact manifold $(N_1,\eta_1)$ into another contact manifold $(N_2,\eta_2)$, is called an \emph{iso-contact} embedding, if $D\iota(TN_1)$ is transverse to $\eta_2$ and $D\iota(TN_1)\cap{\eta_2} = D\iota(\eta_1)$. The above definition then implies that the mapping torus $\mathcal{MT}(\Sigma_1,\phi_1)$ iso-contact embeds into the mapping torus $\mathcal{MT}(\Sigma_2,\phi_2)$. Since $g|_{\partial \Sigma_1}$ pulls back the contact form $\alpha_2$ to $h\cdot\alpha_1$ for some positive function $h$ on $\partial \Sigma_1$, we can extend this embedding to an iso-contact embedding $\mathcal{I}$ of $\mathcal{A}ob(\Sigma_1,d\alpha_1,\phi_1)$ into $\mathcal{A}ob(\Sigma_2,d\alpha_2,\phi_2)$, such that the restriction of $\mathcal{A}ob(\Sigma_2,d\alpha_2,\phi_2)$ on the image of $\mathcal{I}$ gives the supporting contact open book $\mathcal{A}ob(\Sigma_1,d\alpha_1,\phi_1)$ of $(M_1,\xi_1)$. 

\subsection{Dehn-Seidel twist} \label{dehn twist} 

Consider the symplectic structure on the cotangent bundle $(T^*S^n,d\lambda_{can})$. Here, $\lambda_{can}$ is the canonical 1-form on $T^*S^n$. In local coordinates $(x_1,x_2,...,x_{n+1},y_1,y_2,...,y_{n+1})$, $\lambda_{can}$ is given by the form $\sum y_idx_i$. Here, we regard $T^*S^n$ as a submanifold of $\mathbb{R}^{2n+2} \cong \mathbb{R}^{n+1} \times \mathbb{R}^{n+1}$. A point $(\vec{x},\vec{y}) \in \mathbb{R}^{n+1} \times \mathbb{R}^{n+1}$, represents a point in $T^*S^n$ if and only if it satisfies the relations: $|\vec{x}| = 1$ and $\vec{x} \cdot \vec{y} = 0$. Here, $\vec{y} \equiv (y_1,..,y_{n+1})$ and $\vec{x} \equiv (x_1,..,x_{n+1})$. 

\

Let $\sigma_t : T^*S^n \rightarrow T^*S^n$ be a map defined as follows.

\begin{equation*}
\sigma_t(\vec{x},\vec{y}) = \begin{pmatrix}
\cos t&|\vec{y}|^{-1}\sin t\   \\  -|\vec{y}|\sin t&\cos t 
\end{pmatrix}
\begin{pmatrix}
\vec{x} \ \\ \vec{y}
\end{pmatrix}
\end{equation*}

\

For $k \in \mathbb{Z}_{>0}$, let $g_k : [0,\infty) \rightarrow \mathbb{R}$ be a smooth function that satisfies the following properties.

\begin{enumerate}
	\item $g_k(0) = k\pi$ and ${g'_k}(0) < 0$.
	\item Fix $p_0 > 0$. The function $g_k(|\vec{y}|)$ decreases to $0$ at $p_0$ and then remains $0$ for all $\vec{y}$ with $|\vec{y}|>p_0$. See Figure $3$. 
\end{enumerate}
\begin{figure}\label{dehntwistpic}
	
	\begin{tikzpicture}
	
	\draw[thick,->] (0,0) -- (4,0) node[anchor=north east] {$|\vec{y}|$} ;
	\draw[thick,->] (0,0) -- (0,4)node[anchor=north east] {$g_k$}; 
	\draw[very thick] (0,3) parabola  (2,1);
	\draw[very thick](3,0) parabola (2,1);
	\draw[very thick] (3,0) -- (3.5,0);
	\draw (3 cm,1pt) -- (3 cm,-1pt) node[anchor=north] {$p_0$};
	
	\draw (0,3) node[anchor=east] {$k\pi$};
	
	\end{tikzpicture}
	\caption{}	
	
\end{figure}

Now we can define the \textit{positive $k$-fold Dehn-Seidel twist} as follows.

$$\tau_k(\vec{x},\vec{y}) = \begin{cases}
\sigma_{g_k(|\vec{y}|)}(\vec{x},\vec{y})  \  \   for \ \ \vec{y} \neq \vec{0}\\
-Id  \  \ for \ \ \vec{y} = \vec{0}
\end{cases}$$

The Dehn-Seidel twist is a proper symplectomorphism of $T^*S^n$. From Figure $3$, we see that $\tau_k$ has compact support. Therefore, choosing $p_0$ properly, $\tau_k$ can be defined on the unit disk bundle $(DT^*S^n,d\lambda_{can})$, such that it is identity near boundary. In fact, we can choose the support as small as we wish without affecting the symplectic isotopy class of the resulting $\tau_k$. More precisely, let $g_k^1$ and $g_k^2$ be two functions similar to $g_k$ as above. Say, $g_k^1$ has support $p_1$ and $g_k^2$ has support $p_2$. Then
$\sigma_{t g_k^1(|\cdot|) + (1-t) g_k^2(|\cdot|)}$ gives a symplectic isotopy between $\sigma_{g_k^1(|\cdot|)}$ and $\sigma_{g_k^2(|\cdot|)}$.

\

Similarly, for $k<0$, we can define the \textit{negative $k$-fold Dehn-Seidel twist}. For $k = 0$, $\tau_0$ is defined to be the identity map of $DT^*S^n$. Sometimes we may say just \emph{Dehn twist} instead of \emph{Dehn-Seidel twist}.

\begin{exmp}[Contact open book embedding and Dehn-Seidel twist]

	An important open book decomposition of $(S^{2n+1},\xi_{std})$ is given with page $(DT^*S^n,d\lambda^n_{can})$ and monodromy a positive Dehn-Seidel twist. In terms of the coordinates discussed above, the standard inclusion of $S^n$ in $S^{n+k}$ is given by $(x_1,x_2,...,x_{n+1}) \mapsto (x_1,x_2,...,x_{n+1},0,0,...,0)$. This induces a proper symplectic embedding of $(DT^*S^n,d\lambda^n_{can})$ in $(DT^*S^{n+k},d\lambda^{n+k}_{can})$. In general, any embedding between manifolds induces a proper symplectic embedding between the corresponing cotangent bundles. The map can be described in the following way. 
	
	Let $W_1 = DT^*M_1$ and $W_2 = DT^*M_2$. A diffeomorphism $f : M_1 \rightarrow M_2$, induces the diffeomorphism $f_{\#} : W_1 \rightarrow W_2$, given by $f_{\#}(x_1, \rho_1) = (f(x_1), \rho_2)$,
	such that $\rho_1 = df_{x_1}^*\rho_2$. Here, $f_{\#}$ pulls back the canonical $1$-form on $W_2$ to the canonical $1$-form on
	$W_1$. Now, consider an embedding $\iota : M^m \rightarrow N^{m+k}$. Note that $\iota^*(DT^*N) = DT^*\iota(M) \bigoplus \nu^*(\iota)$. Here,
	$\nu^*(\iota) = \{(\iota(x), \eta(\iota(x))) \in DT^*N \ \ | \ \ \iota^*(\iota(x), \eta(\iota(x))) = (x, 0) \in DT^*M \}$. Thus, every $1$-form $(\iota(x), \rho(\iota(x))) \in DT^*N$, over a point $\iota(x) \in N$, can be uniquely decomposed into sum of two forms $(\iota(x), \rho_M(x)) \in DT^*\iota(M)$
	and $(\iota(x), \rho_{\nu}(x)) \in \nu^*(\iota)$. Now, we can define the induced map $\iota_{\#} : DT^*M \rightarrow DT^*N$ by $(x, \rho_0(x)) \mapsto
	(\iota(x), \rho_M(\iota(x)))$, such that $\iota_x^*(\rho_M(\iota(x))) = \rho_0(x)$. It follows from the above discussion that $\iota_{\#}$ pulls back the canonical $1$-form on $DT^*N$ to the canonical $1$-form on $DT^*M$. 
	
	Note that in case of $S^n \subset S^{n+k}$, we can take advantage of the simple coordinate system on $T^*S^m \subset \mathbb{R}^{m+1} \times \mathbb{R}^{m+1}$, to write the induced symplectic embedding on the cotangent bundles. In section $3$, we will take this approach to write down the symplectic embedding in explicit coordinates.
	
	Now, observe that a Dehn-Seidel twist on $DT^*S^{n+k}$ induces a Dehn-Seidel twist on the embedded $DT^*S^n$. Therefore, $(S^{2n+1},\xi_{std}) = \mathcal{A}ob(DT^*S^n,d\lambda^n_{can},\tau_1)$ contact open book embeds in $(S^{2n+2k+1},\xi_{std}) = \mathcal{A}ob(DT^*S^{n+k},d\lambda^{n+k}_{can},\tau_1)$. 
\end{exmp}

\subsection{Stabilization of contact open books and overtwisted contact structure}\label{stabilization}

Let $\pi_i : E_i \rightarrow B^n_i$ be an $n$-disk bundle over $B^n_i$, for $i = 1,2$. Choose a point $x_i$ and a disk neighborhood $D^n_i$ of $x_i$ in $B_i$, such that $\pi_i^{-1}(D^n_i)$ is diffeomorphic to $D^n_i \times D^n$, for $i = 1,2$.

The \emph{plumbing} of $E_1$ and $E_2$, at $(x_1, x_2) \in E_1 \times E_2$, is obtained by identifying $D^n_1 \times D^n$ with $D^n_2 \times D^n$ by the following map and then smoothing the corners.

\begin{alignat*}{2}
D^n_1 \times D^n &\xrightarrow{\chi}& D^n_2 \times D^n \\ (q,p) &\longmapsto& (-q,p) 
\end{alignat*}

We denote the plumbing of $E_1$ and $E_2$ by $E_1 \mathsection E_2$.

Now, consider two copies of the unit disk cotangent bundles of sphere, $(DT^*S^n_1, dp_1 \wedge dq_1)$ and $(DT^*S^n_2, dp_2 \wedge dq_2)$, with the canonical symplectic structures on them. Since locally $\chi^{*}(dp_2 \wedge dq_2) = dp_1 \wedge dq_1$, we get an induced symplectic structure on $DT^*S^n_1 \mathsection  DT^*S^n_2$, denoted by $dp_1 \wedge dq_1 \mathsection dp_2 \wedge dq_2$.

\begin{definition}
	
	Consider an open book decomposition given by $\mathcal{A}ob(DT^*M^n,d\lambda_M,\phi_M)$. We call the modified open book, $\mathcal{A}ob(DT^*M^n \mathsection DT^*S^n,\phi_M \circ \tau_1)$, a \textit{positive stabilization} of $\mathcal{A}ob(DT^*M^n,\phi_M)$. When $\tau_1$ is replaced by $\tau_{-1}$, we call the modified open book, a \textit{negative stabilization} of $\mathcal{A}ob(DT^*M^n,\phi_M)$.	
	
\end{definition}

We should mention that the above definition is not the most general definition of stabilization used in the literature. However, it suffices for the purpose of the present article. In general, consider a contact open book $\mathcal{A}ob(W^{2n}, \omega = d\lambda,\phi)$. Let $S_L^{n-1} \subset (\partial W,\lambda)$ be a Legendrian sphere, i.e., $\lambda|_{S_L^{n-1}} = 0$. Attach a Weinstein $n$-handle $H_{2n}$ to $W$, along $S_L$. Say, $D_L^n \subset W$ is a Lagrangian disk ($d\lambda|_{D_L^n} = 0$) with boundary $S_L$. Let $\mathcal{L}$ be the Lagrangian $n$-sphere formed by $D_L^n$ and the Lagrangian core disk of $H_{2n}$. Let $\tilde{\omega}$ denote the resulting symplectic structure on $W^{2n} \cup H_{2n}$ and let $\tau_{\mathcal{L}}$ denote a positive Dehn-Seidel twist along $\mathcal{L}$. Then, $\mathcal{A}ob(W^{2n} \cup H_{2n},\tilde{\omega}, \phi \circ \tau_{\mathcal{L}})$ is called a positive stabilization of $\mathcal{A}ob(W^{2n}, \omega, \phi)$. It is well known that a positive stabilization does not change the contactomorphism type of the total manifold. For a proof of this fact and more details on Weinstein handles, see \cite{W} and \cite{Ko}.

In \cite{CMP}, Casals, Murphy and Presas gave a characterization of an \emph{overtwisted} contact structure in terms of open books. They showed that every overtwisted contact structure is a negative stabilization of some open book decomposition. In particular, $\mathcal{A}ob(DT^*S^n,d\lambda_{can},\tau_{-1})$ gives an overtwisted contact structure on $S^{2n+1}$. We will denote this overtwisted contact structure by $\xi_{ot}$.

\subsection{Boundary connected sum}\label{bandsum}

Consider two disjoint symplectic manifolds $(W_1,\partial W_1,\omega_1)$ and $(W_2,\partial W_2,\omega_2)$ with convex boundaries. If we attach an Weinstein $1$-handle along two points $w_1 \in \partial W_1$ and $w_2 \in \partial W_2$, then we get the boundary connected sum of $(W_1,\omega_1)$ and $(W_2,\omega_2)$ denoted by $(W_1 \#_b W_2, \omega_1 \#_b \omega_2)$.

Now consider two contact open books $\mathcal{A}ob(\Sigma_1,d\alpha_1,\phi_1)$ and $\mathcal{A}ob(\Sigma_2,d\alpha_2,\phi_2)$. It is known that  $$\mathcal{A}ob(\Sigma_1 \#_b \Sigma_2, d\alpha_1 \#_b d\alpha_2, \phi_1 \circ \phi_2) \cong \mathcal{A}ob(\Sigma_1,d\alpha_1,\phi_1) \# \mathcal{A}ob(\Sigma_2,d\alpha_2,\phi_2)$$ (see section $2.4$ in \cite{DGK}). Here the connected sum, denoted by $\#$, is a contact connected sum (see section $2.2$ in \cite{Ge}).

\section{Proof of Theorems}

The standard inclusion of $S^n$ in $S^{n+1}$ induces a proper symplectic embedding $j_0 : (DT^*S^n,\lambda^n_{can}) \hookrightarrow (DT^*S^{n+1},\lambda^{n+1}_{can})$. As discussed in section \ref{dehn twist}, we consider $DT^*S^n \subset \mathbb{R}^{n+1} \times \mathbb{R}^{n+1}$ and $DT^*S^{n+1} \subset \mathbb{R}^{n+2} \times \mathbb{R}^{n+2}$. So, in terms of coordinates, $j_0(\vec{x},\vec{y}) = (\vec{x},0,\vec{y},0)$. Recall that $(\vec{x},\vec{y}) \in \mathbb{R}^{n+1} \times \mathbb{R}^{n+1}$ satisfies $|\vec{x}| = 1$ and $\vec{x} \cdot \vec{y} = 0$. Let $j_1 : (DT^*S^n,\lambda^n_{can}) \hookrightarrow (DT^*S^{n+1},\lambda^{n+1}_{can})$ be another proper symplectic embedding given by : $j_1(\vec{x},\vec{y}) = (\vec{x},0,\vec{y},g(|\vec{y}|))$. Here, $g$ is a smooth cut-off function as described in Figure $4$. Note that $g$ is supported in $[0,\delta]$ and $g(0) = \epsilon$, $g(\frac{\delta}{2}) > \frac{\epsilon}{2}$.

\begin{figure}\label{isocutoff}
	
	\begin{tikzpicture}
	
	\draw[thick,->] (0,0) -- (4,0) node[anchor=north east] {$|\vec{y}|$} ;
	\draw[thick,->] (0,0) -- (0,4)node[anchor=north east] {$g$}; 
	\draw[very thick] (0,3) parabola  (2,1);
	\draw[very thick](3,0) parabola (2,1);
	\draw[very thick] (3,0) -- (3.5,0);
	\draw (3 cm,2pt) -- (3 cm,-2pt) node[anchor=north] {$\delta$};
	\draw (1.5 cm,2pt) -- (1.5 cm,-2pt) node[anchor=north] {$\frac{\delta}{2}$};
	\draw (0,3) node[anchor=east] {$\epsilon$};
	\draw[dashed] (1.5,0) -- (1.5,1.9);
	\draw[dashed] (1.5,1.9) -- (0,1.9);
	\draw (2pt,1.5 cm) -- (-2pt, 1.5 cm) node[anchor=east] {$\frac{\epsilon}{2}$};
	
	\end{tikzpicture}
	\caption{}	
	
\end{figure}

\

The next two lemmas are the main ingredients to prove our theorems.

\begin{lemma} \label{open book lemma 1}
	$j_0$ is symplectic isotopic to $j_1$.
\end{lemma}	

\begin{proof}	Define $j_t(\vec{x},\vec{y}) = (\vec{x},0,\vec{y},t \cdot g(|\vec{y}|))$. Now, $j^*_t(d\lambda^{n+1}_{can}) = j^*_t(\displaystyle\sum_{i=1}^{n+2} dx_i \wedge dy_i) = \displaystyle\sum_{i=1}^{n+1} dx_i \wedge dy_i + j^*_t(dx_{n+2} \wedge dy_{n+2})$. Since $j^*_t(dx_{n+2}) = 0$, we have $j^*_t(d\lambda^{n+1}_{can}) = d\lambda^n_{can}$.
\end{proof}

Note that the isotopy $j_t$ only moves $T^*S^n$ inside $T^*S^{n+1}$, without inducing any non-trivial monodromy on $T^*S^n$.

In general, let $M^n$ be an oriented closed hyper surface in $S^{n+1}$. The normal bundle is then isomorphic to $M \times \mathbb{R}$. Thus, we get an induced symplectic embedding of $(DT^*M,d\lambda_M)$ in $(DT^*S^{n+1},d\lambda^{n+1}_{can})$, such that a symplectic tubular neighborhood of this embedding is isomorphic to $(DT^*M\times T^*\mathbb{R}^1, d\lambda_M \oplus dx\wedge dy)$. We can now define a similar symplectic isotopy sending $(p,v,0,0) \mapsto (p,v,0, t \cdot g(|v|))$ in $T^*M\times T^*\mathbb{R}^1$.

\

The following lemma is an adaptation of Proposition $4$ in \cite{Au} to our setting. The proof is the same as in \cite{Au}, with slight modifications.  

\begin{lemma}\label{open book lemma 2}
	Let $(V,\partial V,d\lambda_V) $ and $(W,\partial W,d\lambda_W)$ be two exact symplectic manifolds with convex boundaries of dimension $2m$ and $2m+2s$ respectively. Let $\psi_t : (V,\partial V) \rightarrow (W,\partial W)$ be a family of proper symplectic embeddings such that near $\partial V$, $\psi_t = \psi_0$ for all $t \in [0,1]$. Then there is a symplectic isotopy $\Psi_t$ of $(W,\partial W,d\lambda_W)$ such that $\Psi_0 = Id$ and $\Psi_1 \circ \psi_0(V) = \psi_1(V)$.
\end{lemma}

\begin{proof}
	
	Let $V_t$ denote $\psi_t(V)$. Then, $\psi_t \circ \psi_0^{-1}$ gives a family of symplectomorphisms from $(V_0,{d\lambda_W}|_{V_0})$ to $(V_t,{d\lambda_W}|_{V_t})$. Since the symplectic normal bundles to all $V_t$ are isomorphic, using Weinstein symplectic neighborhood theorem, we can extend $\psi_t \circ \psi^{-1}_0$ to a family of symplectomorphisms $L_t : U_0 \rightarrow U_t$. Here, $U_t$ is a small symplectic tubular neighborhood of $V_t$ in $W$. Let $\rho_t : (W,\partial W) \rightarrow (W,\partial W)$ be any family of diffeomorphism extending $L_t$. We can assume that $\rho_t$ is identity near $\partial W$. Let $\omega_t = \rho^*_t(d\lambda_W)$ and $\Omega_t = -\frac{d\omega_t}{dt}$. We want to find vector fields $Y_t$ on $W$ such that $d\iota_{Y_t}\omega_t = \Omega_t$ and $Y_t$ is tangent to $V_0$. Let $\omega = d\lambda_W$. Let $\chi_t$ denote the flow corresponding to $Y_t$. We then have the following.

    $$\frac{d}{dt}(\chi^*_t\rho^*_t\omega) = \chi^*_t(\frac{d}{dt}(\rho^*_t\omega) + L_{Y_t}(\rho^*_t\omega)) = \chi^*_t(-\Omega_t + d\iota_{Y_t}\omega_t) = 0$$.
    
    Thus, $\rho_t \circ \chi_t$ becomes a family of symplectomorphisms of $W$. Let $\alpha_t = \iota_{Y_t}\omega_t$. Equivalently, we have to find a $1$-form $\alpha_t$ on $W$, such that $d\alpha_t = \Omega_t$ and at every point $v\in V_0$, the $\omega_W$-symplectic orthogonal $N_vV_0$ to $T_vV_0$ lies in the kernel of $\alpha_t$. Now, $$\Omega_t = -\frac{d\omega_t}{dt} = -\frac{d}{dt}(\rho^*_td\lambda_W) = -d(\frac{d}{dt}\rho^*_t\lambda_W)$$ So, defining $\beta_t = -\frac{d}{dt}\rho^*_t\lambda_W$ gives $d\beta_t = \Omega_t$. Note that $d\beta_t = 0$ over $U_0$ and since $\rho_t$ is identity near $\partial W$, $\beta_t = 0$ near $\partial U_0 \cap \partial W$. Thus, $\beta_t \in H^1(\overline{U_0},\partial \overline{U_0} \cap \partial W;\mathbb{R})$ and $\beta_t|_{V_0} \in H^1(V_0,\partial V_0;\mathbb{R})$. Let $\pi : U_0 \rightarrow V_0$ be the projection map of the symplectic normal bundle and $i_0 : V_0 \hookrightarrow U_0$ be the zero section. Let $\gamma_t = \pi^*\beta_t|_{V_0}$. By construction, $\gamma_t$ is closed over $U_0$ and for any $x\in V_0$ the normal fiber $N_xV_0$ lies in the kernel of $\gamma_t$. Moreover, the composition $\pi^* \circ i_0^*$ induces the identity map over $H^1(\overline{U_0},\partial\overline{U_0} \cap \partial W;\mathbb{R})$. Thus, $[\gamma_t] = [\beta_t|U_0]$ in $H^1(\overline{U_0},\partial\overline{U_0} \cap \partial W;\mathbb{R})$. Hence, there exists a smooth real valued function $f_t$ over $U_0$ such that $\gamma_t = \beta_t + df_t$ over $U_0$. We extend $f_t$ to some smooth real function $g_t$ over $W$ and define $\alpha_t = \beta_t + dg_t.$ The $1$-forms $\alpha_t$ satisfy $d\alpha_t = \Omega_t$, and since $\alpha_t|{U_0} = \gamma_t$, $N_xV_0$ lies in the kernel of $\alpha_t$, for all $x \in V_0$.

\end{proof}

\begin{proof}[\textbf{Proof of Theorem \ref{1st open book theorem}}]
	Recall, $(S^{2n+3},\xi_{std}) = \mathcal{A}(DT^*S^{n+1}, \tau^{n+1}_1)$. We embed $(DT^*S^n,\lambda^n_{can})$ into $(DT^*S^{n+1},\lambda^{n+1}_{can})$ via $j_0$, and apply a $k$-fold Dehn twist $\tau^{n+1}_k$ on $DT^*S^{n+1}$. This will induce the monodromy $\tau^n_k$ on $DT^*S^n$. Next, using Lemma \ref{open book lemma 1}, we isotope $\tau^n_k(DT^*S^n)$ to $j_1\circ\tau^n_k(DT^*S^n)$ inside  $DT^*S^{n+1}$ .  We now apply an $(l-k)$-fold Dehn twist $\tau_{(l-k)}^{n+1}$ on $DT^*S^{n+1}$ with support $\frac{m_0}{2}$. The number $m_0$ is chosen in the following way. 
	
	As discussed in section \ref{dehn twist}, while defining the Dehn twist, we can make the support of the function $g_k$ smaller, without changing the symplectic isotopy class of the resulting Dehn twist map. Let $S_0^{n+1}$ denote the zero section of $DT^*S^{n+1}$. Let $\mathcal{D} = d(j_1 (DT^*S^n), S_0^{n+1})$ be the distance between $j_1(DT^*S^n)$ and $S_0^{n+1}$ in $\mathbb{R}^{n+2} \times \mathbb{R}^{n+2}$, under the standard euclidean metric. For the
	discussion below, we refer to Figure $4$. Let $A_0 = (\vec{w},\vec{0}) \in S_0^{n+1}$ and $A_1(\vec{y}) = (\vec{x}, 0, \vec{y},g(|\vec{y}|)) \in j_1(DT^*S^n)$. Then the square of the distance between $A_0$ and $A_1(\vec{y})$ is given by $$d(A_0, A_1(\vec{y}))^2 = | w
	- (\vec{x}, 0)|^2 + |\vec{y}|^2 + |g(|\vec{y}|)|^2$$. Note that for $|\vec{y}| \leq \frac{\delta}{2}$, $g(|\vec{y}|) > \frac{\epsilon}{2}$ and therefore, $d(A_0, A_1(\vec{y}))^2 > \frac{\epsilon^2}{4}$. For $|\vec{y}| > \frac{\delta}{2}$, $d(A_0, A_1(\vec{y}))^2 > \frac{\delta^2}{4}$. Thus, $\mathcal{D} > \frac{1}{2} min \{\epsilon, \delta\}$. Define, $m_0 = \frac{1}{2} min \{\epsilon, \delta\}$. 
	
	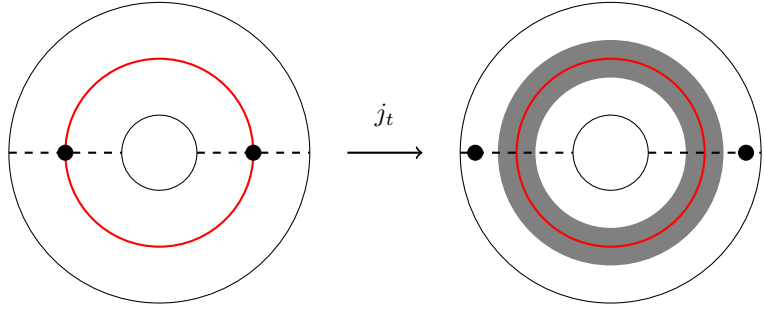
\begin{figure}\label{flex}
		
		\begin{tikzpicture}
		
		\draw[thick,->] (-0.5,0) -- (0.5,0) node[anchor=east] {};
		
		\draw (-3,0) circle (2cm);
		\draw (-3,0) circle (0.5cm);
		\draw[red,thick] (-3,0) circle (1.25);
		\draw[dashed,thick] (-5,0) -- (-3.5,0);
		\draw[dashed,thick] (-2.5,0) -- (-1,0);
		\node[draw,circle,inner sep=2pt,fill] at (-4.25,0) {};
		\node[draw,circle,inner sep=2pt,fill] at (-1.75,0) {};
	
		\draw (3,0) circle (2cm);
		\draw (3,0) circle (0.5cm);
		\fill[gray,even odd rule] (3,0) circle[radius=1cm] circle[radius=1.5cm];
		\draw[red,thick] (3,0) circle (1.25);
		\draw[dashed,thick] (1,0) -- (2.5,0);
		\draw[dashed,thick] (3.5,0) -- (5,0);
		\node[draw,circle,inner sep=2pt,fill] at (1.2,0) {};
		\node[draw,circle,inner sep=2pt,fill] at (4.8,0) {};
		\draw (0,0.5 cm) node[] {$j_t$};

		\end{tikzpicture}
		
		\caption{Using the isotopy $j_t$, we push $DT^*S^0$ away from the zero section $S_0^1$ (denoted by red circles) of $DT^*S^1$. Then we apply Dehn twist along $S_0^1$, with support in the shaded region.}	
		
	\end{figure}

	The above choice of $m_0$ will ensure that $j_1 \circ \tau_k^n(DT^*S^n)$ is not affected by $\tau_{(l-k)}^{n+1}$. Figure $5$ describes the situation for $n=0$. Lastly, we isotope $j_1\circ\tau^n_k(DT^*S^n)$ back to $\tau^n_k(DT^*S^n)$ and finish gluing the mapping torus. Using Lemma \ref{open book lemma 2}, we can extend $j_t$ to a symplectic isotopy $J_t$ of $DT^*S^{n+1}$, such that $J_0 = id$. The resultant monodromy on $DT^*S^{n+1}$ then becomes $J_1^{-1}\circ \tau^{n+1}_{(l-k)}\circ J_1\circ \tau_k^{n+1}$, which is equivalent to getting a mapping torus with monodromy $\tau^{n+1}_l$. When restricted to $DT^*S^n$, it induces the monodromy $\tau^n_k$.

\end{proof}

\begin{proof}[\textbf{Proof of Corollary \ref{1st open book corollary}}]
	By Theorem \ref{1st open book theorem}, $\mathcal{A}ob(DT^*S^n,d\lambda^n_{can},\tau_{-1})$ contact open book embeds in $(S^{2n+3},\eta_{st}) = \mathcal{A}ob(DT^*S^{n+1},d\lambda^{n+1}_{can}, \tau_1)$. By \cite{CMP}, $\mathcal{A}ob(DT^*S^n,d\lambda^n_{can},\tau_{-1})$ gives an overtwisted contact structure on $S^{2n+1}$.
\end{proof}

\begin{proof}[\textbf{Proof of Theorem \ref{2nd open book theorem}}]
	Theorem \ref{1st open book theorem} implies that $\mathcal{A}ob(DT^*S^n,d\lambda^n_{can},\tau_k)$ contact open book embeds in $\mathcal{A}ob(DT^*S^{n+1},d\lambda^{n+1}_{can},\tau_1)$. Following the discussion after Lemma \ref{open book lemma 1} and the proof of Theorem \ref{1st open book theorem}, one can also see that for $V^n \subset S^{n+1}$, $\mathcal{A}ob(DT^*V^n,id)$ contact open book embeds in $\mathcal{A}ob(DT^*S^{n+1},\tau_1)$. Moreover, there is an ambient symplectic isotopy of the identity map of $(DT^*S^{n+1},d\lambda^{n+1}_{can})$, relative to the boundary, that pushes a symplectic neighborhood of the zero section in $DT^*V^n$ away from the zero section of $DT^*S^{n+1}$. Recall that a page of a type-$1$ open book is constructed by taking plumbing and boundary connected sum of such $DT^*V^n$s and $DT^*S^n$s.

	\begin{enumerate}
		\item {\it Boundary connected sum :}
		
		Using the boundary connected sum operation described in section \ref{bandsum}, we get a symplectic embedding of $DT^*M_1^n \#_b DT^*M_2^n$ in $DT^*S_1^{n+1} \#_b DT^*S_2^{n+1}$. Assume that the monodromy map of $DT^*M_1^n \#_b DT^*M_2^n$ is identity. Here, $S_1^{n+1}$ and $S_2^{n+1}$ are used to denote two copies of $S^{n+1}$. Let $\Phi^t_i$ be the ambient symplectic isotopy of $DT^*S_i^{n+1}$ that pushes a symplectic neighborhood of the zero section in $DT^*M_i^n$ away from the zero section of $DT^*S_i^{n+1}$, for $i = 1, 2$. Since $\Phi^t_1$ and $\Phi^t_2$ are identity near boundary for all $t \in [0,1]$, we can extend these isotopies to an isotopy $\tilde{\Phi}^t$ on $DT^*S_1^{n+1} \#_b DT^*S_2^{n+1}$, by defining it identity on the attached $1$-handle of the boundary connected sum. Now, following the proof of Theorem \ref{1st open book theorem}, apply $\tilde{\Phi}^t$ to push away a neighborhood of $M^n_1 \subset DT^*M^n_1$ and $M^n_2 \subset DT^*M^n_2$, inside $DT^*S_1^{n+1}$ and $DT^*S_2^{n+1}$ respectively. Then apply positive Dehn-twists along $S_1^{n+1}$ and $S_2^{n+1}$ with small enough supports. Bring back the neighborhoods of $M^n_1$ and $M^n_2$ by $\tilde{\Phi}^{1-t}$ and complete the mapping torus. The effect of the resultant monodromy on $DT^*M_1^n \#_b DT^*M_2^n$ is identity. On $DT^*S_1^{n+1} \#_b DT^*S_2^{n+1}$, it is the composition of Dehn twists along $S_1^{n+1}$ and $S_2^{n+1}$. Thus, by section \ref{bandsum} the result follows.
		
		\item {\it Plumbing :}
		
		 Consider $DT^*S_1^n \mathsection DT^*S_2^n \hookrightarrow DT^*S_1^{n+1} \mathsection DT^*S_2^{n+1}$. Let $\phi_0$ be a symplectomorphism of $DT^*S_1^n \mathsection DT^*S_2^n$, generated by Dehn twists $\tau^1_1$ and $\tau^2_1$ along $S_1^n$ and $S_2^n$, respectively. For the moment, assume that $\phi_0 = \tau^1_l\circ\tau^2_k$. Let $t \in (0,1)$ denote the $S^1$ direction in the mapping torus of $\mathcal{A}ob(DT^*S_1^n \mathsection DT^*S_2^n,\phi_0)$. Within the $t$-interval $[\frac{1}{4},\frac{1}{3}]$, we apply an $l$-fold Dehn twist $\tau^1_l$ along $S^{n+1}_1$. Next, we isotope $DT^*S_1^n \subset DT^*S_1^{n+1}\subset DT^*S_1^{n+1} \mathsection DT^*S_2^{n+1}$, as in the proof of Theorem \ref{1st open book theorem}, to push it away from the zero section of $DT^*S_1^{n+1}$. We then apply a $(-l+1)$-fold Dehn twist along $S_1^{n+1} \subset DT^*S_1^{n+1}$, with small enough support. Finally, we isotope the pushed away $DT^*S_1^n$ back to its original place. The procedure is similar to the proof of Theorem \ref{1st open book theorem}, but here we extend the isotopy in the complement of $DT^*S_1^n \subset DT^*S_1^n \mathsection DT^*S_2^n$ by the identity map. Next, we perform the same process within the $t$-interval $[\frac{1}{2},\frac{3}{4}]$, starting with a $k$-fold Dehn twist $\tau^2_k$ along $S_2^{n+1}$. This produces an open book embedding of $\mathcal{A}ob(DT^*S_1^n \mathsection DT^*S_2^n,\phi_0)$ in $(S^{2n+3},\eta_{st})$. In general, we can factor the monodromy into Dehn twists along various copies of $S^n$ and then divide the $S^1$-interval of the mapping torus accordingly, to apply the same argument finitely many times.

		\item {\it General case :}
	The general case follows by combining the two cases above.

	\end{enumerate}

\end{proof}

\begin{proof}[\textbf{Proof of Corollary \ref{2nd open book corollary}}]
	
    The surface $\Sigma_g$, as described in Fig \ref{humphreygen}, deformation retracts onto a diffeomorphic surface as in Figure $6$. We can further deformation retract it to a diffeomorphic surface as in Figure $7$, which is obtained by plumbing
	the cotangent bundles of the circles, described by the red curves $b_1, a_1, c_1,...,a_{g-1}, c_{g-1}, a_g$ in Figure \ref{humphreygen}.
	The result now follows from Theorem \ref{2nd open book theorem}.

	\begin{figure}\label{openbookplumb1}
		\begin{center}
			\includegraphics[width=10cm,height=3cm]{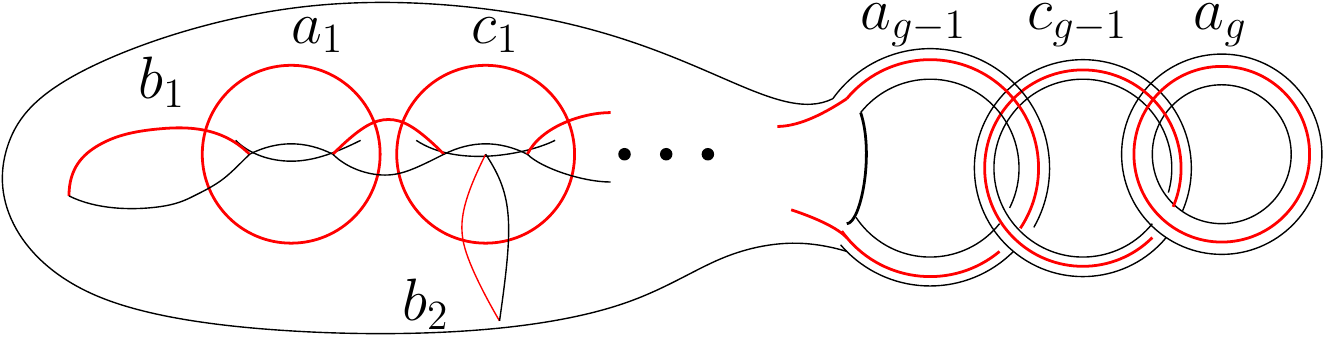}
		\end{center}
		\caption[]{}
	\end{figure}

	\begin{figure}\label{openbookplumb2}
		\begin{center}
			\includegraphics[width=10cm,height=3cm]{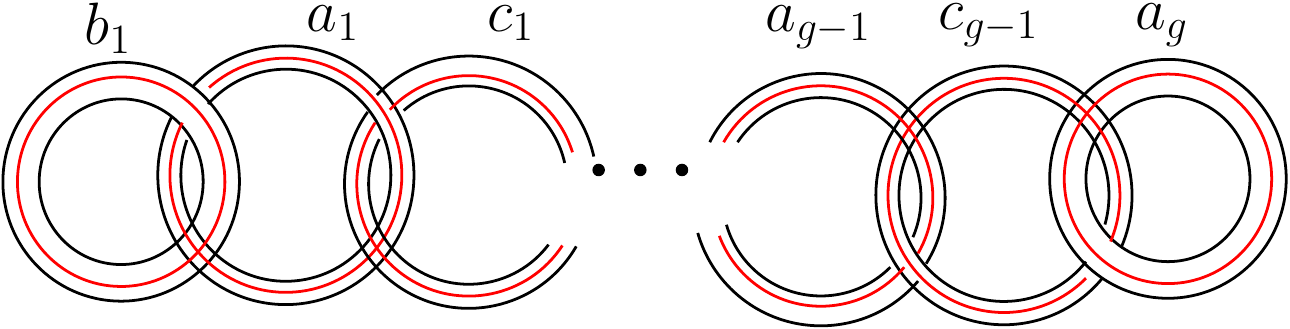}
		\end{center}
		\caption[]{}
	\end{figure}

\end{proof}

\section{comments and remarks} \label{com rem}

\subsection{Contact open book embedding of overtwisted contact 3-manifolds}

Let $\mathcal{A}ob(\Sigma, \phi)$ be a contact open book of
$(M^3, \xi)$, such that $\mathcal{A}ob(\Sigma, \phi)$ contact open book embeds in $(N^5, \eta)$, for some open book $\mathcal{A}ob(W^4, \psi)$ supporting $(N^5, \eta)$. Assume that
$M^3$ has no $2$-torsion in $H_1(M; \mathbb{Z})$. If the first Chern class $c_1(\xi) = 0$, then it is known that the overtwisted contact structures on $M^3$ are in one to one correspondence with the overtwisted contact structures on $S^3$ (see section $2.2$ and $2.3$ of \cite{EF}). The overtwisted contact structures on $S^3$ corresponds to elements in the group $\mathbb{Z}$. They are denoted by $\xi_n$ for $n \in \mathbb{Z}$. For $n \geq 1$, $\xi_n$ is supported by the open book $\mathcal{A}ob(DT^*S^1 \mathsection...n\ \ copies...\mathsection DT^*S^1, \tau_{-1} \circ...n \ \ times...\circ \tau_{-1})$. Thus, by Theorem \ref{1st open book theorem}, $(S^3,\xi_n)$ contact open book embeds in $\mathcal{A}ob(V^2_n,\chi_n) = \mathcal{A}ob(DT^*S^2 \mathsection...n \ \ copies...\mathsection DT^*S^2, \tau_1 \circ...n \ \ times...\circ \tau_1) = (S^5,\xi_{std})$, for $n \geq 1$. Thus, if $c_1(\xi) = 0$, then at least half of the overtwisted contact structures on $M^3$, admit contact open book embedding in $\mathcal{A}ob(W^4 \#_b V^2_n, \psi \circ \chi_n) = (N^5,\eta)$. If we know $W^4$ explicitely, then we get explicit open book embeddings of those overtwisted contact structures on $M^3$.   

For example, consider the real projective space $RP^3$ with the unique tight contact structure $\xi^2$ on it. Let $\xi_n^2$ denote the overtwisted contact structures on $RP^3$, obtained by taking connected sum with $(S^3,\xi_n)$, for $n \geq 1$. By \cite{CM}, $(RP^3,\xi^2) = \mathcal{A}ob(DT^*S^1, \tau_2)$ contact open book embeds in $\mathcal{A}ob(D^4, id) = (S^5,\xi_{std})$. Therefore, $(RP^3,\xi_n^2)$ contact open book embeds in $\mathcal{A}ob(V^2_n, \chi_n) = (S^5,\xi_{std})$.

\subsection{Contact open book embedding of exotic contact spheres}\label{exotic}

Koert and Niederkr\"uger 
\cite{KN} has proved that for $m \geq 2$, all the Ustilovsky spheres of dimension $4m+1$ admit open book decompositions with page $DT^*S^{2m}$ and monodromy a $k$-fold Dehn twist, for some odd $k$. Since, plumbing of pages of the supporting open book gives contact connected sum of the corresponding contact manifolds (see Proposition $2.6$ in \cite{CM}), every contact exotic sphere $\Sigma^{4m+1}$ which is a connected sum of the Ustilovsky spheres, has a contact open book embedding in $(S^{4m+3},\xi_{std})$. In fact, Theorem \ref{1st open book theorem} shows that an Ustilovsky sphere of dimension $4m+1$ contact open book embeds in all the contact homotopy $(4m+3)$-spheres given by $\mathcal{A}ob(DT^*S^{2m+1}, \tau_k)$ for $k \in \mathbb{Z}$.      

\subsection{A class of contact $(2n+1)$-manifolds that contact open book embed in $(S^{4n-1},\xi_{std})$}

 Every oriented closed manifold $V^n$ embeds in $S^{2n-1}$ with a non-zero normal vector field. Thus, we can follow the proof of Theorem \ref{1st open book theorem} to show that the contact manifold $(N^{2n+1},\eta) = \mathcal{A}ob(DT^*V^n, d\lambda_{can}, id)$ contact open book embeds in $(S^{4n-1},\xi_{std}) = \mathcal{A}ob(DT^*S^{2n-1}, \tau_1)$. Note that for $n \geq 4$, $\pi_1(N^{2n+1}) = \pi_1(V^n)$ and given any finitely presented group $\mathcal{G}$, we can find $V^n$ with $\pi_1(V^n) = \mathcal{G}$. Thus, given any finitely presented group $\mathcal{G}$, we can find a contact manifold $(N^{2n+1}, \xi)$ such that $\pi_1(N) = \mathcal{G}$ and $(N^{2n+1}, \xi)$ contact open book embeds in $(S^{4n-1},\xi_{std}) = \mathcal{A}ob(DT^*S^{2n-1}, \tau_1)$. So, in dimension greater than $9$, the fundamental
 group of a contact manifold $(M^{2n+1}, \xi)$ does not pose any obstruction to contact open book embedding in
 $(S^{4n-1}, \xi_{std})$. Kasuya has shown that every $2$-connected contact $(2n+1)$-manifold admits an iso-contact embedding in $(S^{4n+1},\xi_{std})$, for $n \geq 3$ (see Theorem $1.5$ in \cite{Ka}). Moreover, Torres \cite{To} has proved that every $(2n + 1)$-dimensional contact manifold admits a contact open book embedding in $(S^{4n+3},\xi_{std})$.

\subsection{Contact manifolds with first Chern class zero}

As discussed in \cite{Ka2}, a necessary condition for iso-contact embedding of $(M^{2n+1},\xi)$ in $(S^{2n+3},\xi_{std})$ is $c_1(\xi) = 0$. Thus, Theorem \ref{2nd open book theorem} provides a class of contact manifolds with vanishing first Chern class.        	

\subsection{contact open books with page $\Sigma_g$ and monodromy involving $\tau_{b_2}$} \label{nonexample} 

Consider the surface $(\Sigma_g,\partial \Sigma_g)$ of Figure \ref{humphreygen}. Let $\{m_1,m_2,...,m_{2k+1}\}$ be an odd chain of curves on $(\Sigma_g,\partial \Sigma_g)$, in minimal position. Consider a closed regular neighborhood $N(m_1,...,m_{2k+1})$ of their union. The boundary of $N(m_1,...,m_{2k+1})$ then consists of two simple closed curves, $d_1$ and $d_2$. It is known that $(\tau_{m_1}\circ \tau_{m_2}\circ...\circ \tau_{m_{2k+1}})^{2k+2} = \tau_{d_1} \circ \tau_{d_2}$ (see section $4.4$ of \cite{FM}). The map $\tau_{d_1} \circ \tau_{d_2}^{-1}$ is called the \emph{chain map} corresponding to $\{m_1,m_2,...,m_{2k+1}\}$. The $length$ of a chain is the number of simple closed curves in it. A chain is called $odd$ or $even$, depending on the parity of its length. In \cite{J}, Johnson proved the following remarkable fact.

\begin{theorem}[Johnson,\cite{J}]
	For $g \geq 3$, the odd subchain maps of the two chains $\{b_1, a_1, c_1,..., a_{g-1}, c_{g-1}, a_g\}$ and $\{\tau_{b_2}(a_2), c_2, a_3, c_3,...,a_{g-1}, c_{g-1}, a_g\}$ generate the Torelli subgroup of $\mathcal{MCG}(\Sigma_g,\partial \Sigma_g)$.
\end{theorem}

Recall that the Torelli subgroup consists of elements in $\mathcal{MCG}(\Sigma_g,\partial \Sigma_g)$, which has trivial action on $H_1(\Sigma_g, \partial \Sigma_g; \mathbb{Z})$. Now, consider the odd subchain $\{\tau_{b_2}(a_2), c_2, a_3\}$ on $(\Sigma_g,\partial \Sigma_g)$. Let $e_1,e_2$ be the boundary curves of $N(\tau_{b_2}(a_2), c_2, a_3)$. The discussion above then implies the following. $$ \tau_{e_1} \circ \tau_{e_2}^{-1} = (\tau_{\tau_{b_2}(a_2)} \circ \tau_{c_2} \circ \tau_{a_3})^4 \circ \tau_{e_2}^{-2} = (\tau_{b_2}\circ \tau_{a_2} \circ \tau_{b_2}^{-1} \circ \tau_{c_2} \circ \tau_{a_3})^4 \circ \tau_{e_2}^{-2} $$. So, $\tau_{e_1} \circ \tau_{e_2}^{-1}$ has a factorization involving $\tau_{b_2}$. Since $\tau_{e_1} \circ \tau_{e_2}^{-1}$ is an element of the Torelli subgroup, $\mathcal{A}ob(\Sigma_g,\tau_{e_1} \circ \tau_{e_2}^{-1})$ is a homology sphere. Pancholi and Pandit (Theorem $3$ in \cite{PP}) have proved that if a closed $3$-manifold has no $2$-torsion in the first integer homology group, then a contact structure $\eta$ on that manifold admits iso-contact embedding in $(S^5,\xi_{std})$, if and only if $c_1(\eta) = 0$. Therefore, the contact open book $\mathcal{A}ob(\Sigma_g,\tau_{e_1} \circ \tau_{e_2}^{-1})$ admits an iso-contact embedding in $(S^5,\xi_{std})$. Many such examples can be formed by finding elements in the Torelli subgroup that involves $\tau_{b_2}$ in its factorization.


\begin{thebibliography}{References}
    \bibitem[Al]{Al} J. Alexander, A lemma on systems of knotted curves, {\it Proc. Nat. Acad. Sci.}, vol. 9, (1923), 93--95.

	\bibitem[Au]{Au} D. Auroux, Asymptotically homomorphic families of symplectic submanifolds. {\it Geom. Funct. Anal.} 7 (1997), 971-995. 
	
	\bibitem[CE]{CE} R. Casals, J.B. Etnyre, Non-simplicity of isocontact embeddings in all higher dimensions, { \it arXiv:1811.05455}.
	
	\bibitem[CM]{CM} R. Casals, E. Murphy, Contact topology from the loose viewpoint, {\it Proceedings of ${22}^{nd}$ G\"okova Geometry-Topology Conference, 81-115}
	
	\bibitem[CMP]{CMP} R. Casals, E. Murphy, F. Presas, Geometric criteria for overtwistedness, {\it Journal of American Mathematical Society}, 32 (2019), 563-604.
	
	\bibitem[CPP]{CPP} R. Casals, D. Pancholi, F. Presas, The Legendrian Whitney Trick, {\it arXiv:1908.04828}.
	
	\bibitem[DGK]{DGK} F. Ding, H. Geiges, O. Van Koert, Diagrams for contact $5$-manifolds, {\it Journal of the London Mathematical Society}, Volume 86, Issue 3, December 2012, 657–682.
	
	\bibitem[EF]{EF} J. B. Etnyre, R. Furukawa, Braided embeddings of contact $3$-manifolds in the standard contact $5$-sphere, {\it Journal of Topology}, Volume 10, Issue 2, June 2017, 412-446.
	
	\bibitem[EL]{EL} J. Etnyre  and Y. Lekili, Embedding all contact $3$--manifolds in a fixed contact $5$--manifold, {\it Journal of the London Mathematical Society}, Volume 99, Issue 1, February 2019, 52-68.	
	
	\bibitem[Et]{Et}  J. Etnyre, Lectures on open book decompositions and contact structures, 
	{\it Floer homology, gauge theory, and low-dimensional topology} 5,  103--141.


    \bibitem[FM]{FM} B. Farb and D. Margalit, A primer on mapping class groups, Version 5.0, { \it Princeton University Press}.
    
    
    \bibitem[Ge]{Ge} H. Geiges, An Introduction to Contact Topology, {\it Cambridge Studies in Advanced Mathematics} 109.

    \bibitem[Gi]{Gi} E. Giroux, G\'eom\'etrie de contact: de la dimension trois vers les dimensions sup\'erieures, {\it Proceedings of the ICM, Beijing} 2002, vol. 2, 405--414.
     
     
     	
    \bibitem[HH]{HH} K. Honda, Y. Huang, Convex hypersurface theory in contact topology, { \it arXiv:1907.06025}.
     
    \bibitem[J]{J} D. Johnson, The Structure of the Torelli Group I, {\it Ann. of Math.}, Vol. 118, No. 3 (Nov., 1983), pp. 423-442.
     
    \bibitem[Ka]{Ka} N. Kasuya, On contact embeddings of contact manifolds in the odd dimensional Euclidan spaces, { \it International Journal of Mathematics}, vol. 26, No. 7 (2015) 1550045.
     

	\bibitem[Ka2]{Ka2} N. Kasuya, An obstruction for co-dimension two contact embeddings in the odd dimensional Euclidean spaces, {\it J. Math. Soc. Japan} Volume 68, Number 2 (2016), 737-743.
	
	\bibitem[KN]{KN} O. Van Koert, Klaus Niederkr\"uger, Open book decompositions for contact structures on Brieskorn manifolds, {\it Pro. of the American Mathematical Society}, Vol. 133, No. 12, 3679 - 3686.
	
	   
	\bibitem[Ko]{Ko} O. van Koert, Lecture notes on stabilization of contact open books, {\it M\"unster J. Math.}, 10 (2017), 2, 425-455.


	\bibitem[La]{La} T. Lawson, Open book decomposition for odd dimensional manifolds, {\it Topology}, vol 17, (1979), 189--192. 
	
	\bibitem[Mo]{Mo} A. Mori, Global models of contact forms. { \it J. Math. Sci. Univ. Tokyo}, vol. 11(4), 447454, (2004).
	
	\bibitem[PP]{PP} D. M. Pancholi, S. Pandit, Iso-contact embeddings of manifolds in co-dimension $2$, {\it arXiv:1808.04059} (2018).
	
	\bibitem[PPS]{PPS} D. M. Pancholi, S. Pandit, K. Saha, Embeddings of 3-manifolds via open books, { \it arXiv:1806.09784}, 2018.
	
	\bibitem[Qu]{Qu} F. Quinn, Open book decompositions and the bordism of automorphisms, {\it Topology},
	vol. 18 (1), (1979), 55--73.
	

    \bibitem[Ta]{Ta} I. Tamura, Spinnable structures on differentiable manifolds, {\it Proc. Japan Acad.} vol. 48,
    (1972), 293--296.
    
    \bibitem[To]{To} D. Martinez-Torres, Contact embeddings in standard contact spheres via approximately holomorphic geometry, { \it J. Math. Sci. Univ. Tokyo}, vol. 18 (2), 139154.
	
	\bibitem[TW]{TW} W. Thurston, H Winkelnkemper, On the existence of contact forms, {\it Proceedings of AMS} vol. 52, 345 -- 347.

	
	\bibitem[W]{W} A. Weinstein, Contact surgery and symplectic handlebodies, { \it Hokkaido Mathematical Journal}, 2, 241-251,
	(1991).

	
    \bibitem[Wi]{Wi} H. Winkelnkemper,  Manifolds as open books, {\it Bull. Amer. Math. Soc.}, vol. 7, (1973), 45--51.

\end{thebibliography}
\end{document}